\documentclass[11pt]{article}
\usepackage{amsmath,amsmath,amsthm,amssymb}
\usepackage{courier}
\usepackage{graphicx}
\usepackage{enumerate}
\usepackage{fullpage}
\usepackage{hyperref}
\newtheorem{theorem}{Theorem}
\newtheorem*{thm}{Theorem}
\newtheorem{lemma}{Lemma}

\newtheorem{proposition}{Proposition}
\newtheorem{remark}{Remark}        % Use {\rm ...}

\numberwithin{equation}{section}

\begin{document}

\title{A refined approach for non-negative entire solutions of $\Delta u + u^p = 0$ with subcritical Sobolev growth}

\author{John Villavert\footnote{email: john.villavert@gmail.com, john.villavert@utrgv.edu} \\
[0.2cm] {\small University of Texas Rio Grande Valley}\\
{\small Edinburg, Texas 78539, USA}
}
 
\date{}

\maketitle

\begin{abstract} 
Let $N\geq 2$ and $1 < p < (N+2)/(N-2)_{+}$. Consider the Lane-Emden equation $\Delta u + u^p = 0$ in $\mathbb{R}^N$ and recall the classical Liouville type theorem: if $u$ is a non-negative classical solution of the Lane-Emden equation, then $u \equiv 0$. The method of moving planes combined with the Kelvin transform provides an elegant proof of this theorem. A classical approach of J. Serrin and H. Zou, originally used for the Lane-Emden system, yields another proof but only in lower dimensions. Motivated by this, we further refine this approach to find an alternative proof of the Liouville type theorem in all dimensions.
\end{abstract}

 \medskip


\section{Introduction}\label{Introduction}

We revisit the Liouville property of non-negative entire solutions for the Lane-Emden equation
\begin{equation}\label{LE}
  \begin{array}{l}
    \Delta u + u^{p} = 0, ~\, x \in \Omega, 
  \end{array}
\end{equation}
where $N\geq 2$, $1 < p < (N+2)/(N-2)_{+}$ and $\Omega \subseteq \mathbb{R}^N$. By a solution of equation \eqref{LE} we always mean a solution in the classical sense. The Lane-Emden equation arises as a model in astrophysics and has important consequences to fundamental problems in conformal geometry. Specifically, when $\Omega = \mathbb{R}^N$ and $p = (N+2)/(N-2)$, the solvability of the Lane-Emden equation is related to a sharp Sobolev inequality and the classical Yamabe problem \cite{LeeParker87}. Equation \eqref{LE} is also the ``blow-up" equation associated with a family of second-order elliptic equations with Dirichlet data. In other words, the non-existence of positive solutions to the Lane-Emden equation is important in obtaining a priori bounds to positive solutions of elliptic problems within this family \cite{GS81apriori}. The critical Sobolev exponent 
\begin{equation}
p_S := (N+2)/(N-2)_{+} 
= \left\{\begin{array}{c l}
          \infty & \mbox{ if } N = 2, \medskip \\
 	  (N+2)/(N-2) & \mbox{ if } N \geq 3, \\
        \end{array}
 \right.
\end{equation}
plays the role of the dividing number for the existence and non-existence of solutions. More precisely, recall the following celebrated result.
\begin{thm}
Let $N \geq 3$, $p > 1$ and $\Omega = \mathbb{R}^N$. Then equation \eqref{LE} has no positive solution, if and only if $p < p_S$. In particular,
\begin{enumerate}[(a)]
\item if $p = p_S$, then every positive solution of \eqref{LE} has the form
\begin{equation}\label{classify}
 u(x) = c(N)\Big( \frac{t}{t^2 + |x - x_0 |^{2}} \Big)^{(N-2)/2}
\end{equation}
where $c = c(N), t$ are positive constants and $x_0$ is some point in  $\mathbb{R}^N$;

\item if $1 < p < p_S$ and $u$ is a non-negative solution of \eqref{LE}, then necessarily $u \equiv 0$.
\end{enumerate}
\end{thm}
The classification result of part (a) was obtained in \cite{GNN81} under the decay assumption that $u(x) = O(|x|^{2-N})$ (see \cite{Obata71} for an equivalent result), and Caffarelli, Gidas and Spruck in \cite{CGS89} later improved the result by dropping this decay assumption. Part (b) was obtained in \cite{GS81} using local integral estimates and identities. The proofs of both parts were simplified considerably in \cite{CL91} with the aid of the Kelvin transform and the method of moving planes. We refer those who are interested  to \cite{Farina07} and the references therein for more classification and Liouville type theorems for the Lane-Emden equation with respect to various types of solutions, e.g., stable, finite Morse index, radial, or sign-changing, etc.

In this paper, we use a carefully modified approach originated by Serrin and Zou \cite{SZ96} based on integral estimates and a Rellich-Pohozaev identity to prove 
\begin{theorem}\label{rough liouville theorem}
Let $N \geq 2$, $\Omega = \mathbb{R}^N$, $1 < p < p_S$, and suppose $u$ is a non-negative solution of equation \eqref{LE}. Then necessarily $u \equiv 0$.
\end{theorem}

The initial application of Serrin and Zou's technique was to extend this version of the Liouville type theorem to the following Lane-Emden system ($p,q > 0$):
\begin{equation}\label{LE system}
  \begin{array}{l}
    \Delta u + v^{p} = 0, ~\, x \in \Omega, \\
    \Delta v + u^{q} = 0, ~\, x \in \Omega.
  \end{array}
\end{equation}
Obtaining an analogous non-existence theorem for this system, however, proved rather difficult. It remains an open problem and is commonly known as the Lane-Emden conjecture. More precisely, the conjecture infers that system \eqref{LE system} has no positive classical solution in $\Omega = \mathbb{R}^N$, if and only if the subcritical condition
\[ \frac{1}{1 + p} + \frac{1}{1 + q} > 1 - \frac{2}{N} \]
holds. This subcritical condition is indeed a necessary one because of the existence result in \cite{SZ98}, i.e., the Lane-Emden system admits a positive entire solution in the non-subcritical case. Thus, it only remains to verify the sufficiency of the subcritical condition, and this has only been achieved in special cases. The conjecture was proved under radial solutions for any dimension \cite{Mitidieri96}. Then, Serrin and Zou \cite{SZ96} introduced their approach to prove the non-existence of positive solutions in the subcritical case but with the added restriction that solutions have polynomial growth and $N = 3$. Pol\'{a}\v{c}ik, Quittner and Souplet \cite{PQS07} refined the method and result of \cite{SZ96} by removing the growth assumption on solutions, and Souplet \cite{Souplet09} pushed it further to include the case when $N=4$ (for other related partial results, see \cite{BM02,CL09,FiFe94,RZ00}). Hence, the conjecture is completely resolved for dimension $N \leq 4$. Interestingly enough, Souplet \cite{Souplet09} further indicates that their modified approach also applies directly to equation \eqref{LE} but in lower dimensions, of course. In view of this, our goal in the present paper is to further refine the approach to work for the scalar equation in all dimensions. We do this in order to present the underlying obstructions of the method when dealing with such elliptic equations and demonstrate how to possibly overcome them. 

Another motivation for our study centers on the fact that moving plane arguments have their limitations even for related scalar problems. For example, the method of moving planes may not always apply to equations of the form 
$$-\Delta u = c(x)u^p \,\text{ in }\, \Omega = \mathbb{R}^N$$ 
when the coefficient $c(x)$ has certain growth, e.g., $c(x) = |x|^{\sigma}$ with $\sigma > 0$ (see \cite{PhanSouplet12}). Therefore, the improved approach here may prove useful in similar situations. Despite this, the method of moving planes and its variants are powerful tools in studying properties of solutions to many elliptic equations and systems. For more details, the interested reader is referred to the following partial list of papers and the references found therein: \cite{CGS89,CL91,CC97,GNN79,GNN81,HL15,Li96,YYLiZhu95,LiZh03,Polacik11,RZ00,Serrin71} (see also \cite{Gui96,LeiLi16,Ni82} for closely related results on Lane-Emden equations).

Roughly speaking, a key ingredient in our proof of the Liouville type theorem relies on the Rellich-Pohozaev identity. In particular, for a non-negative solution $u$ of the subcritical Lane-Emden equation in 
$\Omega = B_{R}(0) := \{ x \in \mathbb{R}^N : |x| < R\}$ with boundary $\partial B_{R}(0)$, we have 
\begin{equation*}
\Big( \frac{N}{p+1} - \frac{N-2}{2} \Big) \int_{B_{R}(0)} u^{p+1} \,dx 
= \int_{\partial B_{R} (0)} \Big\lbrace R\frac{u^{p+1}}{p+1} + R^{-1}|x\cdot Du|^2 - \frac{R}{2}|D u|^2 + \frac{N-2}{2} u \frac{\partial u}{\partial \nu} \Big\rbrace \,dS,
\end{equation*}
where $\nu = x/|x|$ is the outer unit normal vector at $x \in \partial B_{1}(0)$ and  $\partial u/ \partial \nu = \nabla u \cdot \nu$ is the directional derivative. The idea is to exploit this identity to estimate the energy quantity
$$ F(R) := \int_{B_{R}(0)} u^{p+1} \,dx.$$
Namely, the Rellich-Pohozaev identity implies the estimate
$$ F(R) \leq C( G_{1}(R) + G_{2}(R)),$$
where 
$$ G_{1}(R) := R^{N}\int_{\mathbb{S}^{N-1}} u(R,\theta)^{p+1} \,d\theta,$$
and
$$ G_{2}(R) := R^{N}\int_{\mathbb{S}^{N-1}} \Big( |D u(R,\theta)|^2 + R^{-2}u(R,\theta)^2 \Big) \,d\theta.$$
Here we are writing $u = u(r,\theta)$ in spherical coordinates with $r = |x|$ and $x/|x| \in \mathbb{S}^{N-1}$ ($x \neq 0$), and $D = D_x$ is the gradient operator in terms of the spatial variable $x$. 

\begin{remark}
Hereafter, for functions $f = f(\theta)$ in $L^{s}(\mathbb{S}^{N-1})$, we denote its norm by $ \|f(\theta)\|_{L^{s}(\mathbb{S}^{N-1})}$, and for brevity we write 
\begin{equation}\label{spherical norm}
\|u(r)\|_{s} = \|u(r,\theta)\|_{L^{s}(\mathbb{S}^{N-1})}.
\end{equation}
\end{remark}
It will suffice to prove that equation \eqref{LE} has no positive entire solutions. Therefore, if we assume $u$ is indeed a positive solution, then we shall control the surface integrals above in terms of $F(R)$ to arrive at the feedback estimate $F(R) \leq C[G_{1}(R) + G_{2}(R)] \leq CR^{-a} F(R)^{1-b}$ for some $a,b > 0$. Hence, after taking a sequence $R = R_j \rightarrow \infty$, we may conclude that $\|u\|_{L^{p+1}(\mathbb{R}^N)} = 0$. Thus, $u\equiv 0$, and we arrive at a contradiction. The advantage of this approach is that we effectively remove one degree of freedom in the spatial dimension when estimating the energy quantity. The preceding argument, however, has a bottleneck in the sense that a weaker integrability of $u$ is still required for the procedure to run smoothly. By weaker we mean that $u$ is not assumed a priori to have, say, finite energy or be integrable (see Remark \ref{remarks on energy} below). This is due to certain standard estimates employed, e.g., the Sobolev and interpolation inequalities, which are sensitive to the dimension and are not quite sharp enough to get the feedback estimates for larger $N$. Specifically, for any non-negative solution $u$ of equation \eqref{LE}, we require that there exists a number $q_0 := p + \varepsilon_0$ with
\begin{equation}\label{cond1}
q_0 > (N-1)(p-1)/2
\end{equation}
such that
\begin{equation}\label{cond2} 
\int_{B_{1}(0)} u^{q_0} \,dx \leq C(N,p),
\end{equation}
where $C(N,p) > 0$ depends only on $N$ and $p$. In fact, \eqref{cond2} is known to hold for any $N \geq 2$ whenever $q_0 = p$, i.e.,
\begin{equation}\label{Lp}
\int_{B_{1}(0)} u^p \,dx \leq C(N,p).
\end{equation}
This follows simply by testing equation \eqref{LE} directly by a suitable cut-off function; see, for example, \cite{Souplet09,PhanSouplet12} for the details. Then, it is clear that \eqref{cond1} is also satisfied for $q_0 = p$ so long as $N \leq 4$, and this is precisely where the obstruction on the dimension appears. For $N \geq 5$, estimate \eqref{Lp} is enough to get the non-existence result but only in the non-optimal range $1 < p < (N-1)/(N-3)$, or equivalently,
\begin{equation}\label{non-optimal subcritical}
p > (N-1)(p-1)/2.
\end{equation}
Similar obstructions occur for the Lane-Emden system as indicated earlier by the authors in \cite{CHL16}, which has motivated our work here. Nonetheless, we shall see below that conditions \eqref{cond1} and \eqref{cond2} are necessary and sufficient for the non-existence result.

After the completion of this paper, Professor Souplet kindly sent us his paper \cite{Souplet12} concerning Liouville type theorems for Schr\"{o}dinger elliptic systems with gradient structure and where conditions \eqref{cond1} and \eqref{cond2} also appear. There the conditions ensured that such systems do not admit any non-negative entire solutions having certain exponential growth besides the trivial solution (see Theorem 2 in \cite{Souplet12}).

\begin{remark}\label{remark2}
If $N \geq 3$ and $1 < p \leq N/(N-2)$, then estimating $G_{1}(R)$ and $G_{2}(R)$ is much more direct and the proof of the Liouville type theorem is far simpler in this case. In fact, it is more or less a consequence of \eqref{Lp}. So we assume hereafter that\begin{equation}\label{subcritical range}
\frac{N}{N-2} < p < p_S \,\text{ whenever }\, N \geq 3.
\end{equation}
The case $N=2$ follows in a similar fashion, and it provides the complete non-existence result for $p \in (1,\infty)$. In view of \eqref{non-optimal subcritical}, we also assume that 
\begin{equation}
p \leq (N-1)(p-1)/2 \,\text{ whenever }\, N \geq 5.
\end{equation}
\end{remark}

\begin{remark}
In some of the previous papers utilizing Serrin and Zou's approach, the resulting Liouville type theorems were proved under a boundedness assumption on solutions. Then a ``doubling" property was used to lift this extra condition. However, we make no such boundedness assumptions in this paper. Moreover, the method has been further developed recently and some notable papers are \cite{FYZ14,CHL16,Cowan14,QS12a}. 
\end{remark}
 
%%%%%%%%%%%%%%%%%%%%%%%%%%%%%%%%%%%%%%%%%%%%%%%%%%%%%%%%%%%%%%%%%%%%%%%%%%%%%%%%%%%%%%%%%%%%%%%%%%%%%%%%

\section{Proof of Theorem \ref{rough liouville theorem}}

We provide the details of the main steps for the reader's convenience preferring to state without proof those intermediate results we deem standard. Meanwhile, we hope to provide enough of the details to successfully highlight where exactly the bottleneck of the approach arises. To simplify our presentation, since the main result is known in lower dimensions, we assume $N \geq 5$ unless specified otherwise.

\subsection{Preliminary estimates and scaling invariance}
The following lemmas are elementary (see \cite{CHL16,Souplet09,SZ96}).
\begin{lemma}[Sobolev inequalities on $\mathbb{S}^{N-1}$]
Let $j \geq 1$, $1 \leq z_1 < z_2 \leq \infty$, and $z_1 \neq (N-1)/j$. Then 
$$W^{j,z_1}(\mathbb{S}^{N-1}) \hookrightarrow L^{z_2}(\mathbb{S}^{N-1}),$$ 
where
\begin{equation*}
  \left\{\begin{array}{ll}
    \frac{1}{z_2} = \frac{1}{z_1} - \frac{j}{N-1}, & \mbox{ if } z_1 < \frac{N-1}{j}; \medskip \\
    z_2 = \infty, & \mbox{ if } z_1 > \frac{N-1}{j}.
  \end{array}
\right.
\end{equation*}
Namely, for any $u = u(\theta) \in W^{j,z_1}(\mathbb{S}^{N-1})$, there exists a positive constant $C=C(j,z_1,N)$ such that 
$$ \|u\|_{L^{z_2}(\mathbb{S}^{N-1})} \leq C\Big( \|D^{j}_{\theta} u\|_{L^{z_1}(\mathbb{S}^{N-1})} + \|u\|_{L^{1}(\mathbb{S}^{N-1})} \Big).$$
\end{lemma}

\begin{lemma}[Basic elliptic and interpolation inequalities]\label{interpolation}
Let $1 < k < \infty$ and $R > 0$. For $u \in W^{2,k}(B_{2R}(0))$, there exists a positive constant $C = C(k,N)$ such that
\begin{equation}\label{interpolation1}
 \|D_{x}^{2}u\|_{L^{k}(B_{R}(0))} \leq C\Big( \|\Delta u\|_{L^{k}(B_{2R}(0))} + R^{\frac{N}{k} - (N+2)}\|u\|_{L^{1}(B_{2R}(0))} \Big)
 \end{equation}
and
\begin{equation}\label{interpolation2}
 \|D_{x}u\|_{L^{1}(B_{R}(0))} \leq C\Big( R^{N(1- 1/k) + 1}\|D_{x}^{2} u\|_{L^{k}(B_{2R}(0))} + R^{-1}\|u\|_{L^{1}(B_{2R}(0))} \Big).
 \end{equation}
\end{lemma}

\begin{lemma}[Scaling invariance]\label{scaling invariance}
Let $u$ be a positive solution of equation \eqref{LE} with $\Omega = \mathbb{R}^N$, and assume that $\|u\|_{L^{k}(B_{1}(0))}^k \leq C(k,N,p)$ for some $1 \leq k < \infty$. Then
\begin{equation}\label{dilated Lq energy}
\int_{B_{R}(0)} u^{k} \,dx \leq C(k,N,p)R^{N-k\frac{2}{p-1}}.
\end{equation}
\end{lemma}

\begin{proof}
It is simple to check that equation \eqref{LE} is scaling invariant under the transformation $$u_{\lambda}(x) = \lambda^{\frac{2}{p-1}}u(\lambda x), ~\, \lambda > 0,$$
i.e., for any $\lambda > 0$, $u_{\lambda}$ remains a solution provided $u$ is a solution. Then, by a change of variables and scaling with $\lambda = R$, we obtain 
\begin{align*}
\int_{B_{R}(0)} u(x)^{k} \,dx = {} & \int_{B_{1}(0)} u(Ry)^k R^N \,dy = R^{N - k\frac{2}{p-1}} \int_{B_{1}(0)} u_{R}(y)^k \,dy \\
\leq {} & C(k,N,p)R^{N - k\frac{2}{p-1}}.
\end{align*}

\end{proof} 

\subsection{A local integral estimate}
As discussed earlier, an important ingredient in proving the Liouville type theorem relies on a key local integral estimate. Namely, we require
\begin{proposition}\label{energy prop}
Let $N \geq 2$, $\Omega = \mathbb{R}^N$ and suppose that $u$ is any non-negative solution of equation \eqref{LE} where $1 < p < p_S$. Then there exists a $q_0 > 0$ satisfying \eqref{cond1} for which \eqref{cond2} holds.
\end{proposition}

\begin{remark}\label{remarks on energy}
\mbox{ }
\begin{enumerate}[(a)] 

\item If $q_0 = p+1$, then such solutions are said to have finite energy. If $p = p_S$, then finite energy solutions have the form \eqref{classify} due to a variant of the method of moving planes adapted for integral equations \cite{CLO06}. 

\item If $q_0 = N(p-1)/2$, then $u$ is sometimes called an integrable solution. In fact, the integrable solutions are equivalent to the finite energy solutions and the classification and Liouville type theorems for such solutions have been established in \cite{CLO05a,CLO06}. If $q_0 > N(p-1)/2$ in Proposition \ref{energy prop}, then the Liouville type theorem follows directly from the scaling invariance of the equation.
\end{enumerate}

\end{remark}

\begin{proof}[Proof of Proposition \ref{energy prop}] This is already known for some $q_0 > N(p-1)/2$. The proof uses Bochner's integral formula tested carefully against a suitably chosen cut-off function, and we refer the reader to \cite{GS81} and Chapter 8 of \cite{QS07} for the details. For $q_0 \leq N(p-1)/2$, the local estimate follows from potential theory and the symmetry and monotonicity of non-negative entire solutions. Namely, we can assume entire solutions are monotone decreasing about the origin by either moving plane \cite{CL91} or moving sphere \cite{YYLiZhu95,LiZh03} arguments. Then, by the integral representation of solutions, we deduce that
$$ |x|u(x)^{(p-1)/2} \leq C(N,p), ~ x \in B_{1}(0),$$
since
\begin{align*}
u(x) \geq {} & C\int_{\mathbb{R}^N} |x - y|^{2-N}u(y)^p \,dy \geq C \int_{B_{|x|}(0)} |x - y|^{2-N}u(y)^p \,dy \\
\geq {} & Cu(x)^p \int_{B_{|x|}(0)} (1 + |y|^2)^{(2-N)/2} \,dy \geq C u(x)^p |x|^2.
\end{align*}
This implies 
$$ \int_{B_{1}(0)} u^{\gamma(p-1)/2} \,dx \leq C(N,p) < \infty$$
provided that $\gamma \leq N$. We choose $q_0 = \gamma(p-1)/2$ with $\gamma \in (N-1, N]$ so that \eqref{cond1} and \eqref{cond2} are satisfied. This completes the proof.

\end{proof}

\begin{remark}
It would be more interesting, however, to obtain the local integral estimate without relying on the monotonicity properties of positive solutions to equation \eqref{LE}.
\end{remark}

\subsection{Intermediate inequalities for estimating $G_{1}(R)$ and $G_{2}(R)$}

Let $u$ be a positive solution of equation \eqref{LE} in $\Omega = \mathbb{R}^N$. We set $\tau = 2/(p-1)$ and $k = (p+1)/p$. Due to Proposition \ref{energy prop} and Remarks \ref{remark2} and \ref{remarks on energy}, we may choose $q \in (p,q_0)$ and $\ell > 1$ such that 
$$(N-1)(p-1)/2 < q := \ell p < q_0 = p + \varepsilon_0 < p+1.$$ 
Therefore, $\ell < k$ and 
\begin{equation}\label{q energy}
\int_{B_{1}(0)} u^q \,dx \leq C(N,p). 
\end{equation}
Recalling \eqref{spherical norm}, we may write 
\[ \|u\|_{L^{s}(B_{R}(0))}^{s} = \int_{0}^{R} \|u(r)\|_{s}^{s} r^{N-1} \,dr.\]

\begin{lemma}\label{ball estimates}
For large $R \geq 1$ and all suitably small $\epsilon \in (0, \varepsilon_0)$, the following hold:
\begin{enumerate}[(a)]
\item $\|u\|_{L^{1}(B_{R}(0) )} \leq CR^{N-\tau}$,
\item $\|D_{x} u\|_{L^{1}(B_{R}(0) )} \leq CR^{N + 1 - p\tau}$,
\item $\|D_{x}^2 u\|_{L^{\ell + \epsilon}(B_{R}(0) )}^{\ell + \epsilon} \leq CR^{N - q\tau }$,
\item $\|D_{x}^{2} u\|_{L^{k}(B_{R}(0) )}^{k} \leq C F(2R)$.
\end{enumerate}
\end{lemma}

\begin{proof}
\noindent (a) By H\"{o}lder's inequality, taking $k = q$ in Lemma \ref{scaling invariance} and combining it with estimate \eqref{q energy}, we get
\begin{equation*}
\int_{B_{R}(0) } u \,dx \leq \Big( \int_{B_{R}(0)} 1 \,dx \Big)^{1 - 1/q} \Big( \int_{B_{R}(0)} u^q \,dx \Big)^{1/q} 
\leq C R^{N(1-1/q)} R^{(N - q\tau )/q}  = CR^{N - \tau }.
\end{equation*}

\noindent (b) From estimate \eqref{interpolation1} of Lemma \ref{interpolation} combined with \eqref{q energy} and Lemma \ref{scaling invariance}, we get

\begin{align*}
\|D_{x}^2 u\|_{L^{1+ \epsilon}(B_{R}(0))} \leq {} & C( \|\Delta u\|_{L^{1+\epsilon}(B_{2R}(0))} + R^{\frac{N}{1+\epsilon} - (N+2)}\|u\|_{L^{1}(B_{2R}(0))} ) \\
\leq {} & C( \|u\|_{L^{p(1+\epsilon)}(B_{2R}(0))}^p + R^{\frac{N}{1+\epsilon} - (N+2)}\|u\|_{L^{1}(B_{2R}(0))} ) \\
\leq {} & C( (R^{N - p(1+\epsilon)\tau})^{\frac{p}{p(1+\epsilon)} } + R^{\frac{N}{1+\epsilon} - (N+2)}\|u\|_{L^{1}(B_{2R}(0))} ) \\
\leq {} & C( R^{\frac{N}{1+\epsilon} - p\tau } + R^{\frac{N}{1+\epsilon} - (N+2)}\|u\|_{L^{1}(B_{R}(0))} ).
\end{align*}

Inserting this into estimate \eqref{interpolation2} of Lemma \ref{interpolation} and invoking the estimate of part (a), we get
\begin{align*}
\|D_{x}u\|_{L^{1}(B_{R}(0) )} \leq {} & C( R^{N(1-\frac{1}{1+\epsilon}) + 1}\|D^2 u\|_{L^{1+\epsilon}(B_{2R}(0) )} + R^{-1}\|u\|_{L^{1}(B_{2R}(0) )} ) \\
\leq {} & C ( R^{N(1-\frac{1}{1+\epsilon}) + 1} ( R^{\frac{N}{1+\epsilon} - p\tau } + R^{\frac{N}{1+\epsilon} - (N+2)}\|u\|_{L^{1}(B_{2R}(0))} ) + R^{-1}\|u\|_{L^{1}(B_{2R}(0) )}  ) \\
\leq {} & C ( R^{-1 + N - \tau} + 2R^{-1}\|u\|_{L^{1}(B_{2R}(0) )}  ) \leq C(R^{N+1 - (2 +\tau)} + R^{N + 1 - p\tau } ) \\
\leq {} & C R^{N + 1 - p\tau }.
\end{align*}

\noindent (c) With $k = \ell + \epsilon$ in estimate \eqref{interpolation1} of Lemma \ref{interpolation}, we have 
\begin{align}\label{lemma4c1}
\|D_{x}^2 u\|_{L^{\ell + \epsilon}(B_{R}(0) )}^{\ell + \epsilon} \leq {} & C( \|\Delta u\|_{L^{\ell + \epsilon}(B_{2R}(0) )}^{\ell + \epsilon} + R^{N-(\ell + \epsilon)(N+2)}\|u\|_{L^{1}(B_{2R}(0) )}^{\ell + \epsilon} ) \notag \\
\leq {} & C( \int_{B_{2R}(0)} u^{p(\ell + \epsilon)}\,dx + R^{N-(\ell + \epsilon)(N+2)}R^{(\ell + \epsilon)(N-\tau)} ). 
\end{align}
Since $q < p(\ell + \epsilon) < q_0$ for sufficiently small $\epsilon$, interpolation ensures that
$$\|u\|_{L^{p(\ell + \epsilon)} (B_{1}(0) )}^{p(\ell + \epsilon)} \leq C(N,p).$$ 
Thus, combining this with Lemma \ref{scaling invariance} and inserting the resulting estimate into \eqref{lemma4c1} yields
\begin{equation*}
\|D_{x}^2 u\|_{L^{\ell + \epsilon}(B_{R}(0) )}^{\ell + \epsilon} \leq C(R^{N-p(\ell + \epsilon)\tau} + R^{N-(\ell + \epsilon)(2+\tau)} ) \leq CR^{N-p(\ell + \epsilon)\tau} \leq CR^{N-q\tau}.
\end{equation*}

\noindent (d) By Lemma \ref{interpolation}, part (a), and the fact that $N -(p+1)\tau < 0$ due to $p < p_S$, we obtain
\begin{align*}
\|D_{x}^2 u\|_{L^{k}(B_{R}(0) )}^{k} \leq {} & C \Big( \int_{B_{2R}(0) } |\Delta u|^k \,dx + R^{N-k(N+2)}( \int_{B_{2R}(0) } u\,dx )^{k} \Big) \\
\leq {} & C\Big (  \int_{B_{2R}(0) } u^{pk} \,dx + R^{N-k(N+2)}R^{k(N-\tau)} \Big ) \\
%\leq {} & C( F(2R) + R^{N-k(N+2)}R^{k(N-\tau)}) \\
\leq {} & C( F(2R) + R^{N- (2+\tau)k)} \leq C( F(2R) + R^{ N - (p\tau)k } ) \\
\leq {} & C( F(2R) + R^{N - (p+1)\tau} ) \leq C F(2R).
\end{align*}
This completes the proof of the lemma.
\end{proof}

With the help of Lemma \ref{ball estimates}, we establish

\begin{proposition}\label{measure feedback}
For each large $R \geq 1$ and for all suitably small $\epsilon \in (0, \varepsilon_0)$, there exists $\tilde{R} \in (R,2R)$ such that the following hold:
\begin{enumerate}[(a)]
\item $\|u(\tilde{R})\|_{1} \leq CR^{-\tau}$,
\item $\|D_{x}u(\tilde{R})\|_{1} \leq CR^{1 - p\tau }$,
\item $\|D_{x}^2 u(\tilde{R})\|_{\ell + \epsilon } \leq CR^{- \frac{q\tau}{\ell + \epsilon} }$,
\item $\|D_{x}^2 u(\tilde{R})\|_{k} \leq C \Big( R^{-N} F(2R) \Big )^{\frac{1}{k} }$.
\end{enumerate}
\end{proposition}

\begin{proof}
Fix a large $R \geq 1$. Given a constant $K > 0$, set
\begin{align*}
\Gamma_{1}(R) = {} & \Big \lbrace r \in (R,2R) \, : \, \|u(r)\|_{1} > KR^{-\tau} \Big\rbrace, \\
\Gamma_{2}(R) = {} & \Big \lbrace r \in (R,2R) \, : \, \|D_{x}u(r)\|_{1 } > KR^{1 - p\tau } \Big\rbrace, \\
\Gamma_{3}(R) = {} & \Big \lbrace r \in (R,2R) \, : \, \|D_{x}^2 u(r)\|_{\ell + \epsilon}^{ \ell + \epsilon} > KR^{-q\tau } \Big\rbrace, \\
\Gamma_{4}(R) = {} & \Big \lbrace r \in (R,2R) \, : \, \|D_{x}^2 u(r)\|_{k}^k > KR^{-N}F(2R) \Big\rbrace. 
\end{align*} 
We claim that we can choose a large $K$ independent of $R$ such that
\[ \Gamma := (R,2R) \backslash \bigcup_{i=1}^{4} \Gamma_{i}(R) \neq \emptyset. \]
Then it follows that we can find $\tilde{R} \in (R,2R)$ satisfying the estimates stated in the proposition. Therefore, It only remains to prove the claim.
From part (a) of Lemma \ref{ball estimates}, we obtain $|\Gamma_{1}(R)| \leq CR/K$ since
\begin{align*}
C \geq {} & R^{\tau - N}\|u\|_{L^{1}(B_{2R}(0))} = R^{\tau - N} \int_{0}^{2R} \|u(r,\theta)\|_{L^{1}(\mathbb{S}^{N-1})} r^{N-1} \,dr \\
\geq {} & R^{\tau - N}|\Gamma_{1}(R)|R^{N-1}KR^{-\tau} \geq KR^{-1}|\Gamma_{1}(R)|.
\end{align*}
From part (b) of Lemma \ref{ball estimates}, we obtain $|\Gamma_{2}(R)| \leq CR/K$ since
\begin{equation*}
C \geq R^{ p\tau - N - 1} \|D_{x}u\|_{L^{1}(B_{2R}(0))} \geq R^{ p\tau - N - 1} |\Gamma_{2}(R)| R^{N-1}KR^{1 - p\tau } \geq KR^{-1}|\Gamma_{2}(R)|.
\end{equation*}
In a similar manner, we can show that $|\Gamma_{3}(R)|, |\Gamma_{4}(R)| \leq CR/K$ from parts (c) and (d) of Lemma \ref{ball estimates}. By choosing $K \geq 8C$, we get that $|\Gamma_{i}(R)| \leq R/8$ for $i = 1,2,3,4$ and thus 
$$\bigcup_{i=1}^{4} \Gamma_{i}(R) \leq R/2.$$
Hence, $| \Gamma | \geq |(R,2R)| - R/2 = R/2$ and $\Gamma$ is non-empty. 
\end{proof}

%%%%%%%%%%%%%%%%%%%%%%%%%%%%%%%%%%%%%%%%%%%%%%%%%%%%%%%%%%%%%%%%%%%%%%%%%%%%%%%%%%%%%%%%%%%%%%%%%%%%%%%%

\subsection{Proof of Theorem \ref{rough liouville theorem}}
\begin{proof} Let $u$ be a positive solution of \eqref{LE} in $\Omega = \mathbb{R}^N$, and choose $q$ and $\ell$ as before. If $q > N(p-1)/2$, then we are done since Lemma \ref{scaling invariance} implies $u \equiv 0$ after sending $R \rightarrow \infty$ in \eqref{dilated Lq energy}. So hereafter, we assume $q \leq N(p-1)/2$. Moreover, note that \eqref{subcritical range} implies that
$$\ell < k = 1 + 1/p < 1 + (N-2)/N = 2(N-1)/N \leq (N-1)/2$$
so that
$$ \frac{1}{\ell} > \frac{1}{k} > \frac{2}{N-1}.$$

\subsubsection*{Step 1: Estimation of $G_{1}(R)$.}

\noindent Case 1: $\displaystyle \frac{1}{\ell} < \frac{2}{N-1} + \frac{1}{1+p}$.

\medskip
Indeed
$$\frac{1}{\ell} - \frac{2}{N-1} < \frac{1}{1+p},$$
and the Sobolev embedding implies that 
$$ W^{2,\ell + \epsilon}(\mathbb{S}^{N-1}) \hookrightarrow L^{p+1}(\mathbb{S}^{N-1})$$
where 
$$ \|u(R)\|_{p+1} \leq C\Big( R^2 \|D_{x}^2 u(R)\|_{\ell + \epsilon} + \|u(R)\|_{1}\Big).$$
Hence, Proposition \ref{measure feedback} implies the existence of an $\tilde{R} \in (R,2R)$ such that
\begin{align*}
G_{1}(\tilde{R}) \leq {} & C\tilde{R}^{N}\Big( \tilde{R}^2\|D_{x}^2 u(\tilde{R})\|_{\ell + \epsilon} + \|u(\tilde{R})\|_{1} \Big)^{p+1} \\
\leq {} & C R^{N}\Big( R^{2 - \frac{q\tau}{\ell + \epsilon}} + R^{-\tau} \Big)^{p+1} \leq C R^{(p+1)(2 - \frac{q\tau}{\ell + \epsilon} + \frac{N}{p+1})} \\
\leq {} & CR^{(p+1)(2 - p\tau + \frac{N}{p+1} + p\tau - \frac{q\tau}{\ell + \epsilon})} \\
\leq {} & CR^{(p+1)(-\tau + \frac{N}{p+1} + \frac{\epsilon p\tau}{\ell + \epsilon}) },
\end{align*}
where in the last line we used the fact that $2 - p\tau = -\tau$. By noticing that $p < p_S$ implies that $-\tau + N/(p+1) < 0$, we can then choose $\epsilon$ suitably small so that 
$$ G_{1}(\tilde{R}) \leq CR^{-a} \,\text{ for some }\, a > 0.$$

\noindent Case 2: $\displaystyle \frac{1}{\ell} \geq \frac{2}{N-1} + \frac{1}{1+p}$.

Set
$$\frac{1}{\mu} = \frac{1}{k} - \frac{2}{N-1}$$
and 
$$\frac{1}{\lambda} = \frac{1}{\ell} - \frac{2}{N-1} \geq \frac{1}{1+p},$$
and so
$$\frac{1}{\mu} = \frac{1}{k} - \frac{2}{N-1}\leq \frac{1}{1+p},$$
i.e., $1/\mu \leq 1/(1+p) \leq 1/\lambda$. In view of the Sobolev embeddings
$$ W^{2,\ell + \epsilon}(\mathbb{S}^{N-1}) \hookrightarrow L^{\lambda}(\mathbb{S}^{N-1}) \,\text{ and }\, W^{2,k}(\mathbb{S}^{N-1}) \hookrightarrow L^{\mu}(\mathbb{S}^{N-1}),$$
interpolation ensures there is some $\theta \in [0,1]$ with 
$$\frac{1}{1+p} = \frac{\theta}{\lambda} + \frac{1-\theta}{\mu}$$
such that
\begin{align*}
\|u(R)\|_{p+1} \leq {} & \|u(R)\|_{\lambda}^{\theta}\|u(R)\|_{\mu}^{1-\theta} \\
\leq {} & C\Big( R^2 \|D_{x}^2 u(R)\|_{\ell + \epsilon} + \|u(R)\|_{1} \Big)^{\theta} \Big( R^2 \|D_{x}^2 u(R)\|_{k} + \|u(R)\|_{1} \Big)^{1-\theta}.
\end{align*}
Thus, by Proposition \ref{measure feedback}, there exists an $\tilde{R} \in (R,2R)$ such that
\begin{align*}
G_{1}(\tilde{R}) \leq {} & CR^{N}\Big( (R^{2 - \frac{q\tau}{\ell + \epsilon}} + R^{-\tau} )^{\theta} ( R^{2 - \frac{N}{k} }F(2R)^{\frac{1}{k}} + R^{-\tau})^{1-\theta} \Big)^{p+1} \\
\leq {} & CR^{[N/(1+p) + \theta(2 -  \frac{q\tau}{\ell + \epsilon}) + (2 - N/k)(1-\theta)](1+p)}F(2R)^{p(1-\theta)} \\
\leq {} & CR^{-a_{1}(\epsilon)}F(2R)^{1-b_1}, 
\end{align*}
where 
$$a_{1}(\epsilon) :=  [\theta( \frac{q\tau}{\ell + \epsilon} - 2 ) - N/(1+p) - (2 - N/k)(1-\theta)](1+p) ~ \,\text{ and }\, b_1 := 1 - p(1-\theta).$$
Direct calculations yield 
\begin{align*}
a_{1}(0) = {} & -[N/(1+p) + \theta(2 - p \tau) + (2 - N/k)(1-\theta)](1+p) \\
= {} & -[N - \theta(1+p)\tau + (2(1+p) - Np )(1-\theta)] \\
= {} & -[N - \theta(1+p)\tau + 2(1-\theta) - p(1-\theta)(N-2)] \\
= {} & -N + p \tau\theta + \tau\theta - (p\tau - \tau)(1-\theta) + p(1-\theta)(N - 2) \\
= {} & -N + p\tau\theta + \tau\theta - p\tau(1-\theta) + \tau - \tau \theta + p(1-\theta)(N-2) \\
= {} & -N + \underbrace{(2 + \tau - p\tau)}_{= \,0 } + p\tau \theta + \tau -p\tau(1-\theta) + p(1-\theta)(N-2)  \\ %{} & ~\, (\text{since } 2 + \tau - p\tau = 0) \\
= {} & 2\tau + 2 -N - 2p(1-\theta)\tau + p(1-\theta)(N-2) \\
= {} & 2[1 - p(1-\theta)](\tau - (N-2)/2) \\
= {} & 2(\tau - (N-2)/2) b_1,
\end{align*}
where we used the fact that $2 + \tau = p\tau$ and the subcritical condition \eqref{subcritical range}, which ensures that $\tau - (N-2)/2 > 0$. Hence, $a_{1}(0) > 0$ if $b_1 > 0$ and thus $a_{1}(\epsilon) > 0$ for sufficiently small $\epsilon > 0$. Therefore, the desired estimate for this case follows once we verify that
\begin{equation}\label{b1}
b_1 > 0,\,\text{ i.e., }\, p < \frac{1}{1-\theta},
\end{equation}
and we shall see that this holds due to $q > (N-1)(p-1)/2$. 

From the fact that
$$ \frac{1}{1+p} = \frac{\theta}{\ell} + \frac{1-\theta}{k} - \frac{2}{N-1},$$
we see that $\frac{1}{p} =  1 - \theta + \theta\frac{k}{\ell} - \frac{2k}{N-1}$. Therefore, property \eqref{b1} follows immediately once we show that
$$ \frac{\theta}{\ell} > \frac{2}{N-1}.$$
Since 
$$\frac{1}{1+p} = \frac{\theta}{\lambda} + \frac{1-\theta}{\mu} = (1-\theta)\Big(\frac{1}{\mu} - \frac{1}{\lambda}\Big) + \frac{1}{\lambda},$$
we have that 
$$\frac{1}{1-\theta} \Big( \frac{1}{\lambda} - \frac{1}{1+p} \Big) = \frac{1}{\lambda} - \frac{1}{\mu} = \frac{1}{\ell} - \frac{1}{k}.$$
Thus,
$$\frac{1}{(1-\theta)p} = \frac{ \frac{1}{\ell} - \frac{1}{k}}{\frac{p}{\lambda} - \frac{1}{k}} = \frac{ \frac{p}{q} - \frac{1}{k}}{\frac{p}{\lambda} - \frac{1}{k}} > 1,$$
where the right-most inequality follows from 
$$ \frac{1}{q} > \frac{1}{\lambda} = \frac{p}{q} - \frac{2}{N-1},$$
which itself follows from the fact that $q > (N-1)(p-1)/2$. This completes the estimation of $G_{1}(R)$.

\subsubsection*{Step 2: Estimation of $G_{2}(R)$.}
For $z,z' > 1$ with $1/z + 1/z'  = 1$, H\"{o}lder's inequality, Cauchy's inequality, and the Sobolev inequality
$$R^{-1}\|u(R)\|_{z} \leq C(\|D_{x}u(R)\|_{z} + R^{-1}\|u(R)\|_{1})$$ 
imply
\begin{align}\label{G2} 
G_{2}(R) \leq {} & CR^{N} (\|D_{x}u(R)\|_{2} + R^{-1}\|u(R)\|_{2} )^2  \\
\leq {} & CR^{N} ( \|D_{x}u(R)\|_{2}^2 + 2 R^{-1}\|D_{x}u(R)\|_{2}\|u(R)\|_{2} + R^{-2}\|u(R)\|_{2}^2 ) \notag \\
\leq {} & CR^{N} ( \|D_{x}u(R)\|_{2}^2 + R^{-2}\|u(R)\|_{2}^2 ) \notag \\
\leq {} & CR^{N} ( \|D_{x}u(R)\|_{z}\|D_{x}u(R)\|_{z'} + R^{-2}\|u(R)\|_{z}\|u(R)\|_{z'} ) \notag \\
\leq {} & CR^{N}(\|D_{x}u(R)\|_{z} + R^{-1}\|u(R)\|_{z})(\|D_{x}u(R)\|_{z'} + R^{-1}\|u(R)\|_{z'}) \notag \\
\leq {} & CR^{N}(\|D_{x}u(R)\|_{z} + R^{-1}\|u(R)\|_{1})(\|D_{x}u(R)\|_{z'} + R^{-1}\|u(R)\|_{1}) \notag \\
\leq {} & CR^{N+2}(R^{-1}\|D_{x}u(R)\|_{z} + R^{-2}\|u(R)\|_{1})(R^{-1}\|D_{x}u(R)\|_{z'} + R^{-2}\|u(R)\|_{1}).  \notag
\end{align}

%We claim that there exists a $z > 0$ such that
%CHANGE ABOVE LINE WITH 

\noindent It is straightforward to check that we can find $z > 1$ such that
\begin{equation}\label{z}
\max\Big( \frac{1}{k} - \frac{1}{N-1}, \frac{1}{N-1} \Big) \leq \frac{1}{z} \leq \min\Big( \frac{1}{\ell} - \frac{1}{N-1}, \frac{1}{1+p} + \frac{1}{N-1} \Big),
\end{equation}
and thus its H\"{o}lder conjugate $z'$ satisfies
\[ \max \Big( \frac{1}{k} - \frac{1}{N-1}, \frac{\ell - 1}{\ell} + \frac{1}{N-1} \Big) \leq \frac{1}{z'} = 1 - \frac{1}{z} \leq \min \Big( 1 - \frac{1}{N-1}, \frac{1}{1+p} + \frac{1}{N-1} \Big ).\]
In view of this and \eqref{z}, we consider the following embeddings:
$$ W^{1,\ell + \epsilon}(\mathbb{S}^{N-1}) \hookrightarrow L^{\lambda_1}(\mathbb{S}^{N-1}), \, W^{1,k}(\mathbb{S}^{N-1}) \hookrightarrow L^{\mu}(\mathbb{S}^{N-1}), $$
$$ W^{1,1+\epsilon}(\mathbb{S}^{N-1}) \hookrightarrow L^{\lambda_2}(\mathbb{S}^{N-1}),$$
where
$$ \frac{1}{\lambda_1} = \frac{1}{\ell} - \frac{1}{N-1}, \,~\, \frac{1}{\mu} = \frac{1}{k} - \frac{1}{N-1} \,\text{ and }\, \frac{1}{\lambda_2} = 1 - \frac{1}{N-1},$$
and $\ell$ and $k$ are defined as before. Interpolation ensures there exist
$\theta_1,\theta_2 \in [0,1]$ with
\begin{equation}\label{theta}
\frac{1}{z} = \frac{\theta_1}{\lambda_1} + \frac{1-\theta_1}{\mu} \,\text{ and }\, \frac{1}{z'} = \frac{\theta_2}{\lambda_2} + \frac{1-\theta_2}{\mu}
\end{equation}
such that
\begin{align}\label{estimate1}
\|D_{x}u(R)\|_{z} \leq {} & \|D_{x}u(R)\|_{\lambda_1}^{\theta_1} \|D_{x}u(R)\|_{\mu}^{1-\theta_1}  \\
\leq {} & C\Big(R\|D_{x}^{2}u(R)\|_{\ell + \epsilon} + \|D_{x}u(R)\|_{1}\Big)^{\theta_1} \Big( R\|D_{x}^{2}u(R)\|_{k} + \|D_{x}u(R)\|_{1} \Big)^{1-\theta_1}, \notag
\end{align}
\begin{align}\label{estimate2}
\|D_{x}u(R)\|_{z'} \leq {} & \|D_{x}u(R)\|_{\lambda_2}^{\theta_2} \|D_{x}u(R)\|_{\mu}^{1-\theta_2}  \\
\leq {} & C\Big(R\|D_{x}^{2}u(R)\|_{1 + \epsilon} + \|D_{x}u(R)\|_{1}\Big)^{\theta_2} \Big( R\|D_{x}^{2}u(R)\|_{k} + \|D_{x}u(R)\|_{1} \Big)^{1-\theta_2}. \notag
\end{align}
Inserting estimates \eqref{estimate1} and \eqref{estimate2} into \eqref{G2} yields
\begin{align*}
G_{2}(R) \leq {} & CR^{N+2}\Big( \|D_{x}^2 u(R)\|_{\ell + \epsilon} + R^{-1}\|D_{x}u(R)\|_{1} + R^{-2}\|u(R)\|_{1} \Big)^{\theta_1} \\
{} & \times \Big( \|D_{x}^2 u(R)\|_{k} + R^{-1}\|D_{x}u\|_{1} + R^{-2}\|u(R)\|_{1} \Big)^{1-\theta_1} \\
{} & \times \Big( \|D_{x}^2 u(R)\|_{1 + \epsilon} + R^{-1}\|D_{x}u(R)\|_{1} + R^{-2}\|u(R)\|_{1} \Big)^{\theta_2} \\
{} & \times \Big( \|D_{x}^2 u(R)\|_{k} + R^{-1}\|D_{x}u(R)\|_{1} + R^{-2}\|u(R)\|_{1} \Big)^{1-\theta_2}. 
\end{align*}
Thus, by Proposition \ref{measure feedback}, there exists an $\tilde{R} \in (R,2R)$ such that
\begin{align*}
G_{2}(\tilde{R}) \leq {} & CR^{N+2 - \frac{p\tau}{1 + \epsilon}(\theta_1 + \theta_2) - \frac{N}{k}(2-\theta_1 - \theta_2)}F(2R)^{\frac{2-\theta_1 - \theta_2}{k}} \\
\leq {} & C R^{-a_{2}(\epsilon)} F(2R)^{1- b_2},
\end{align*}
where 
$$a_{2}(\epsilon) := (\frac{p\tau}{1 + \epsilon} )(\theta_1 + \theta_2) + \frac{N}{k}(2-\theta_1 - \theta_2) - (N + 2) ~\,\text{ and }\, 1- b_2 := \frac{2-\theta_1 - \theta_2}{k}.$$
To obtain the desired estimate for this case, we shall verify that $b_2 > 0$, then we use this to show that $a_{2}(0) > 0$. Thus, $a_{2}(\epsilon)>0$ for sufficiently small $\epsilon>0$. To see why $a_{2}(0) > 0$, simple calculations yield
\begin{align*}
a_{2}(0) = {} & p\tau  k(b_2 - 1) + 2 p\tau  + N(1-b_2) - (N+2) \\
= {} & (N - p \tau k)(1-b_2) - (N+2) + 2 p\tau \\
= {} & ( N - \tau (p+1) )(1-b_2) - (N+2) + 2 p\tau \\
= {} & (\tau (p+1) - N)b_2 + N - \tau (p+1) - (N+2) + 2 p\tau \\
= {} & (\tau (p+1) - N)b_2 - 2 p\tau  + 2 p\tau ~ (\text{since } \tau (p+1) + 2 = 2p \tau) \\
= {} & (p+1)(\tau - N/(p+1))b_2,
\end{align*}
where the last line is positive due to the subcritical condition \eqref{subcritical range} so long as $b_2 > 0$. Thus, let us now show $b_2 > 0$. From \eqref{theta} we have that
\begin{align*}
1 = {} & \frac{1}{z} + \frac{1}{z'} = \frac{\theta_1}{\lambda_1} + \frac{\theta_2}{\lambda_2} + \frac{2-\theta_1 - \theta_2}{\mu} \\
= {} & \theta_1 (\frac{1}{\ell} - \frac{1}{N-1}) + \theta_2 (1 - \frac{1}{N-1}) + (2-\theta_1 - \theta_2)(\frac{1}{k} - \frac{1}{N-1}).
\end{align*}
Then, from this, we obtain that
$$b_2 = 1 - \frac{2-\theta_1 - \theta_2}{k} = \frac{\theta_1}{\ell} + \theta_2 - \frac{2}{N-1}.$$
So we only need to verify that
$$\frac{\theta_1}{\ell} + \theta_2 > \frac{2}{N-1},$$ 
and we do so by considering two cases due to \eqref{z}.

\medskip

In the case where $$\dfrac{1}{\ell} - \dfrac{1}{N-1} \leq \dfrac{1}{1+p} + \dfrac{1}{N-1},$$
we can take $\frac{1}{z} = \frac{1}{\ell} - \frac{1}{N-1}$. Inserting this into \eqref{theta} yields $\theta_1 = 1$. Indeed, 
\begin{align*}
\frac{\theta_1}{\ell} + \theta_2 - \frac{2}{N-1}  \geq \frac{\theta_1}{\ell} - \frac{2}{N-1} = \frac{1}{\ell} - \frac{2}{N-1} > \frac{1}{k} - \frac{2}{N-1} > 0.
\end{align*}

In the case where $$ \dfrac{1}{1+p} + \dfrac{1}{N-1} < \dfrac{1}{\ell} - \dfrac{1}{N-1},$$
we can set $1/z = 1/(1+p) + 1/(N-1)$. Inserting this into \eqref{theta} shows that $1/z' = 1/\mu$ and so $\theta_{2} = 0$. Therefore, it suffices to verify that 
\begin{equation}\label{this}
\frac{\theta_1}{\ell} > \frac{2}{N-1}.
\end{equation}
From \eqref{theta} we also deduce that
\begin{equation}\label{caseii}
\frac{1}{1+p} = \frac{\theta_1}{\ell} + \frac{1-\theta_1}{k} - \frac{2}{N-1},
\end{equation}
and from \eqref{caseii} we obtain
\[ \frac{\theta_1}{\ell} - \frac{2}{N-1} = \frac{1}{1 + p} - \frac{p(1-\theta_1)}{1+p}, \]
which indicates that \eqref{this} holds if and only if $p < 1/(1-\theta_1)$. Now, to show $p(1-\theta_1) < 1$, we first note that \eqref{caseii} is equivalent to
$$\frac{1}{\lambda_1} - \Big( \frac{1}{1+p} + \frac{1}{N-1} \Big) = p(1-\theta_1) \Big(\frac{1}{p \lambda_1} - \frac{1}{p \mu}\Big).$$
Hence, from this we arrive at
\begin{align*}
p(1-\theta_1) = {} & \frac{\frac{p}{q} - \frac{1}{N-1} - \Big( \frac{1}{1+p} + \frac{1}{N-1}\Big)}{\frac{1}{q} - \frac{1}{1+p}} \\
= {} & \frac{\frac{p}{q} - \frac{1}{1+p} - \frac{2}{N-1}}{\frac{1}{q} - \frac{1}{1+p} } =  \frac{\frac{p-1}{q} + \frac{1}{q} - \frac{1}{1+p} - \frac{2}{N-1}}{\frac{1}{q} - \frac{1}{1+p} } < 1,
\end{align*}
where the inequality in the last line is true because
$$\frac{p-1}{q} - \frac{2}{N-1} < 0 \,\text{ or }\, q > (N-1)(p-1)/2.$$
This completes the estimation of $G_{2}(R)$.

%THIS CLAIM WAS REMOVED ABOVE!!!!
%
%Finally, it only remains to verify the last claim, i.e., there exists a $z>0$ satisfying \eqref{z}. This holds once we verify the following.
%\begin{enumerate}[(i)]
%\item $\frac{1}{k} - \frac{1}{N-1} \leq \frac{1}{\ell} - \frac{1}{N-1}$,
%\item $\frac{1}{N-1} \leq \frac{1}{\ell} - \frac{1}{N-1}$,
%\item $\frac{1}{N-1} \leq \frac{1}{1+p} + \frac{1}{N-1}$,
%\item $\frac{1}{k} - \frac{1}{N-1} \leq \frac{1}{1+p} + \frac{1}{N-1}.$
%\end{enumerate}
%
%Indeed, part (iii) is obvious, part (i) is due to $\ell < k$, and (ii) follows since if $N \geq 4$, then $\frac{1}{\ell} > \frac{1}{k} > \frac{2}{N-1}$. Lastly, part (iv) follows from condition \eqref{subcritical range} with $N \geq 4$.

\subsubsection*{Step 3:} From Steps 1 and 2, we may fix $a, b > 0$ and choose a strictly increasing sequence $\{R_j\}_{j=1}^{\infty}$ such that
$$ F(R_j)^b \leq CR_{j}^{-a}$$
which, after sending $R_j \rightarrow \infty$, implies that $\|u\|_{L^{p+1}(\mathbb{R}^N)} = 0$. Hence, $u\equiv 0$ in $\mathbb{R}^N$, and we arrive at a contradiction. This completes the proof of the theorem.
\end{proof}

\vspace*{5mm}

\noindent{\bf Acknowledgements:} Part of this work was completed while the author participated in the 2016 Summer Research Program at the Mathematical Sciences Research Institute (MSRI). The author wishes to thank MSRI for their hospitality and support during his visit. 

\bibliographystyle{plain}

%\bibliographystyle{alpha}

%\bibliography{refs.bib}

\begin{thebibliography}{10}

%\bibitem{FYZ14}
%F.~Arthur, X.~Yan, and M.~Zhao.
%\newblock A {L}iouville-type theorem for higher order elliptic systems.
%\newblock {\em Discrete Contin. Dyn. Syst.}, 34(9):3317--3339, 2014.
%
%\bibitem{BM02}
%J.~Busca and R.~Man\'{a}sevich.
%\newblock A {L}iouville-type theorem for {L}ane--{E}mden systems.
%\newblock {\em Indiana Univ. Math. J.}, 51:37--51, 2002.
%
%\bibitem{CGS89}
%L.~Caffarelli, B.~Gidas, and J.~Spruck.
%\newblock Asymptotic symmetry and local behavior of semilinear elliptic
%  equations with critical {S}obolev growth.
%\newblock {\em Comm. Pure Appl. Math.}, 42:271--297, 1989.
%
%\bibitem{CHL16}
%Z.~Cheng, G.~Huang, and C.~Li.
%\newblock A {L}iouville theorem for subcritical {L}ane-{E}mden system.
%\newblock {\em Preprint}.
%\newblock arxiv.org/abs/1412.7275.
%
%\bibitem{CL91}
%W.~Chen and C.~Li.
%\newblock Classification of solutions of some nonlinear elliptic equations.
%\newblock {\em Duke Math. J.}, 63:615--622, 1991.
%
%\bibitem{CLO05a}
%W.~Chen, C.~Li, and B.~Ou.
%\newblock Qualitative properties of solutions for an integral equation.
%\newblock {\em Discrete Contin. Dyn. Syst.}, 12:347--354, 2005.
%
%\bibitem{CLO06}
%W.~Chen, C.~Li, and B.~Ou.
%\newblock Classification of solutions for an integral equation.
%\newblock {\em Comm. Pure Appl. Math.}, 59:330--343, 2006.
%
%\bibitem{CL09}
%W.~Chen and C.~Li.
%\newblock An integral system and the {L}ane-Emden conjecture.
%\newblock {\em Discrete Contin. Dyn. Syst.}, 24(4):1167--1184, 2009.
%
%\bibitem{CC97}
%C. C.~Chen and C. S.~Lin.
%\newblock Estimates of the conformal scalar curvature equation via the method of moving planes.
%\newblock {\em Comm. Pure Appl. Math.}, 50(10):971--1017, 1997.
%
%\bibitem{Cowan14}
%C.~Cowan.
%\newblock A {L}iouville theorem for a fourth order {H}enon equation.
%\newblock {\em Adv. Nonlinear Stud.}, 14(3):767--776, 2014.
%
%\bibitem{FiFe94}
%D. G.~de Figueiredo and P. Felmer.
%\newblock A {L}iouville-type theorem for elliptic systems.
%\newblock {\em Ann. Sc. Norm. Super. Pisa. Cl. Sci.}, 21(4):387--397, 1994.
%
%\bibitem{Farina07}
%A.~Farina.
%\newblock On the classification of solutions of the {L}ane-{E}mden equation on unbounded domains of $\mathbb{R}^n$.
%\newblock {\em J. Math. Pures Appl.}, 87(5):537--561, 2007.
%
%\bibitem{GNN79}
%B.~Gidas, W. M.~Ni, and L.~Nirenberg.
%\newblock Symmetry and related properties via the maximum principle.
%\newblock {\em Comm. Math. Phys.}, 68:209--243, 1979.
%
%\bibitem{GNN81}
%B.~Gidas, W. M.~Ni, and L.~Nirenberg.
%\newblock Symmetry of positive solutions of nonlinear elliptic equations in ${R}^n$.
%\newblock {\em Adv. Math. Suppl. Studies A}, 7:369--402, 1981.
%
%\bibitem{GS81}
%B.~Gidas and J.~Spruck.
%\newblock Global and local behavior of positive solutions of nonlinear elliptic equations.
%\newblock {\em Comm. Pure and Appl. Math.}, 34(4):525--598, 1981.
%
%\bibitem{GS81apriori}
%B.~Gidas and J.~Spruck.
%\newblock A priori bounds for positive solutions of nonlinear elliptic equations.
%\newblock {\em Comm. Partial Differential Equations}, 6(8):883--901, 1981.
%
%\bibitem{Gui96}
%C.~Gui.
%\newblock On positive entire solutions of the elliptic equation $\Delta u + K(x)u^p = 0$ and its applications to Riemannian geometry.
%\newblock {\em Proc. Roy. Soc. Edinburgh Sect. A}, 126(2):225--237, 1996.
%
%\bibitem{HL15}
%G.~Huang and C.~Li.
%\newblock A {L}iouville theorem for high order degenerate elliptic equations.
%\newblock {\em J. Differential Equations}, 258(4):1229--1251, 2015.
%
%\bibitem{LeeParker87}
%J.~M. Lee and T.~H. Parker.
%\newblock The {Y}amabe problem.
%\newblock {\em Bull. Amer. Math. Soc. (N. S.)}, 17(1):37--91, 1987.
%
%\bibitem{LeiLi16}
%Y~Lei and C~Li.
%\newblock Sharp criteria of {L}iouville type for some nonlinear systems.
%\newblock {\em Discrete Contin. Dyn. Syst.}, 36(6):3277--3315, 2016.
%
%\bibitem{Li96}
%C.~Li.
%\newblock Local asymptotic symmetry of singular solutions to nonlinear elliptic equations.
%\newblock {\em Invent. Math.}, 123:221--231, 1996.
%
%\bibitem{YYLiZhu95}
%Y.~Y. Li and M.~Zhu.
%\newblock Uniqueness theorems through the method of moving spheres.
%\newblock {\em Duke Math. J.}, 80(2):383--418, 1995.
%
%\bibitem{LiZh03}
%Y.~Y. Li and L.~Zhang.
%\newblock Liouville-type theorems and {H}arnack-type inequalities for semilinear elliptic equations.
%\newblock {\em J. Anal. Math.}, 90:27--87, 2003.
%
%\bibitem{Mitidieri96}
%E.~Mitidieri.
%\newblock Nonexistence of positive solutions of semilinear elliptic systems in ${R}^{N}$.
%\newblock {\em Differ. Integral Equations}, 9:465--480, 1996.
%
%\bibitem{Ni82}
%W. M.~Ni.
%\newblock On the elliptic equation $\Delta u + K(x)u^{(n+2)/(n-2)} = 0$, its generalizations, and applications in geometry.
%\newblock {\em Indiana Univ. Math. J.}, 31(4):493--529, 1982.
%
%\bibitem{Obata71}
%M.~Obata.
%\newblock The conjectures on conformal transformations of Riemannian manifolds.
%\newblock {\em J. Differential Geometry}, 6:247--258, 1971/1972.
%
%\bibitem{Polacik11}
%P.~Pol\'{a}\v{c}ik.
%\newblock Symmetry of nonnegative solutions of elliptic equations via a result of {S}errin.
%\newblock {\em Comm. Partial Differential Equations}, 36(4):657--669, 2011.
%
%\bibitem{PQS07}
%P.~Pol{\'a}{\v{c}}ik, P.~Quittner, and Ph. Souplet.
%\newblock Singularity and decay estimates in superlinear problems via {L}iouville-type theorems, {I}: {E}lliptic equations and systems.
%\newblock {\em Duke Math. J.}, 139(3):555--579, 2007.
%
%\bibitem{PhanSouplet12}
%Q.~H. Phan and Ph. Souplet.
%\newblock Liouville-type theorems and bounds of solutions of {H}ardy--{H}\'{e}non equations.
%\newblock {\em J. Differential Equations}, 252:2544--2562, 2012.
%
%\bibitem{QS07}
%P.~Quittner and Ph. Souplet.
%\newblock {\em {\bf Superlinear Parabolic Problems:} Blow-up, Global Existence and Steady States}.
%\newblock Birkh\"{a}user Verlag, 2007.
%
%\bibitem{QS12a}
%P.~Quittner and Ph. Souplet.
%\newblock Optimal {L}iouville-type theorems for noncooperative elliptic {S}chr\"{o}dinger systems and applications.
%\newblock {\em Comm. Math. Phys.}, 311(1):1--19, 2012.
%
%\bibitem{RZ00}
%W.~Reichel and H.~Zou.
%\newblock Non-existence results for semilinear cooperative elliptic systems via moving spheres.
%\newblock {\em J. Differential Equations}, 161(1):219--243, 2000.
%
%\bibitem{Serrin71}
%J.~Serrin.
%\newblock A symmetry problem in potential theory.
%\newblock {\em Arch. Rational Mech. Anal.}, 43:304--318, 1971.
%
%\bibitem{Souplet09}
%Ph. Souplet.
%\newblock The proof of the {L}ane--{E}mden conjecture in four space dimensions.
%\newblock {\em Adv. Math.}, 221(5):1409--1427, 2009.
%
%\bibitem{Souplet12}
%Ph. Souplet.
%\newblock Liouville-type theorems for elliptic {S}chr\"{o}dinger systems associated with copositive matrices.
%\newblock {\em Netw. Heterog. Media}, 7(4):967--988, 2012.
%
%\bibitem{SZ96}
%J.~Serrin and H.~Zou.
%\newblock Non-existence of positive solutions of {L}ane--{E}mden systems.
%\newblock {\em Differ. Integral Equations}, 9(4):635--653, 1996.
%
%\bibitem{SZ98}
%J.~Serrin and H.~Zou.
%\newblock Existence of positive solutions of the {L}ane--{E}mden system.
%\newblock {\em Atti Sem. Mat. Fis. Univ. Modena}, 46:369--380, 1998.


%%%%%%%%%%%%%%%%

\bibitem[AYZ14]{FYZ14}
F.~Arthur, X.~Yan, and M.~Zhao.
\newblock A {L}iouville-type theorem for higher order elliptic systems.
\newblock {\em Discrete Contin. Dyn. Syst.}, 34(9):3317--3339, 2014.

\bibitem[BM02]{BM02}
J.~Busca and R.~Man\'{a}sevich.
\newblock A {L}iouville-type theorem for {L}ane--{E}mden systems.
\newblock {\em Indiana Univ. Math. J.}, 51:37--51, 2002.

\bibitem[CGS89]{CGS89}
L.~Caffarelli, B.~Gidas, and J.~Spruck.
\newblock Asymptotic symmetry and local behavior of semilinear elliptic
  equations with critical {S}obolev growth.
\newblock {\em Comm. Pure Appl. Math.}, 42:271--297, 1989.

\bibitem[CHL16]{CHL16}
Z.~Cheng, G.~Huang, and C.~Li.
\newblock A {L}iouville theorem for subcritical {L}ane-{E}mden system.
\newblock {\em Preprint}.
\newblock arxiv.org/abs/1412.7275.

\bibitem[CL91]{CL91}
W.~Chen and C.~Li.
\newblock Classification of solutions of some nonlinear elliptic equations.
\newblock {\em Duke Math. J.}, 63:615--622, 1991.

\bibitem[CLO05]{CLO05a}
W.~Chen, C.~Li, and B.~Ou.
\newblock Qualitative properties of solutions for an integral equation.
\newblock {\em Discrete Contin. Dyn. Syst.}, 12:347--354, 2005.

\bibitem[CLO06]{CLO06}
W.~Chen, C.~Li, and B.~Ou.
\newblock Classification of solutions for an integral equation.
\newblock {\em Comm. Pure Appl. Math.}, 59:330--343, 2006.

\bibitem[CL09]{CL09}
W.~Chen and C.~Li.
\newblock An integral system and the {L}ane-Emden conjecture.
\newblock {\em Discrete Contin. Dyn. Syst.}, 24(4):1167--1184, 2009.

\bibitem[CL97]{CC97}
C. C.~Chen and C. S.~Lin.
\newblock Estimates of the conformal scalar curvature equation via the method of moving planes.
\newblock {\em Comm. Pure Appl. Math.}, 50(10):971--1017, 1997.

\bibitem[COW14]{Cowan14}
C.~Cowan.
\newblock A {L}iouville theorem for a fourth order {H}enon equation.
\newblock {\em Adv. Nonlinear Stud.}, 14(3):767--776, 2014.

\bibitem[DFF94]{FiFe94}
D. G.~de Figueiredo and P. Felmer.
\newblock A {L}iouville-type theorem for elliptic systems.
\newblock {\em Ann. Sc. Norm. Super. Pisa. Cl. Sci.}, 21(4):387--397, 1994.

\bibitem[FAR07]{Farina07}
A.~Farina.
\newblock On the classification of solutions of the {L}ane-{E}mden equation on unbounded domains of $\mathbb{R}^n$.
\newblock {\em J. Math. Pures Appl.}, 87(5):537--561, 2007.

\bibitem[GNN79]{GNN79}
B.~Gidas, W. M.~Ni, and L.~Nirenberg.
\newblock Symmetry and related properties via the maximum principle.
\newblock {\em Comm. Math. Phys.}, 68:209--243, 1979.

\bibitem[GNN81]{GNN81}
B.~Gidas, W. M.~Ni, and L.~Nirenberg.
\newblock Symmetry of positive solutions of nonlinear elliptic equations in ${R}^n$.
\newblock {\em Adv. Math. Suppl. Studies A}, 7:369--402, 1981.

\bibitem[GS81a]{GS81}
B.~Gidas and J.~Spruck.
\newblock Global and local behavior of positive solutions of nonlinear elliptic equations.
\newblock {\em Comm. Pure and Appl. Math.}, 34(4):525--598, 1981.

\bibitem[GS81b]{GS81apriori}
B.~Gidas and J.~Spruck.
\newblock A priori bounds for positive solutions of nonlinear elliptic equations.
\newblock {\em Comm. Partial Differential Equations}, 6(8):883--901, 1981.

\bibitem[GUI96]{Gui96}
C.~Gui.
\newblock On positive entire solutions of the elliptic equation $\Delta u + K(x)u^p = 0$ and its applications to Riemannian geometry.
\newblock {\em Proc. Roy. Soc. Edinburgh Sect. A}, 126(2):225--237, 1996.

\bibitem[HL15]{HL15}
G.~Huang and C.~Li.
\newblock A {L}iouville theorem for high order degenerate elliptic equations.
\newblock {\em J. Differential Equations}, 258(4):1229--1251, 2015.

\bibitem[LP87]{LeeParker87}
J.~M. Lee and T.~H. Parker.
\newblock The {Y}amabe problem.
\newblock {\em Bull. Amer. Math. Soc. (N. S.)}, 17(1):37--91, 1987.

\bibitem[LL16]{LeiLi16}
Y~Lei and C~Li.
\newblock Sharp criteria of {L}iouville type for some nonlinear systems.
\newblock {\em Discrete Contin. Dyn. Syst.}, 36(6):3277--3315, 2016.

\bibitem[LI96]{Li96}
C.~Li.
\newblock Local asymptotic symmetry of singular solutions to nonlinear elliptic equations.
\newblock {\em Invent. Math.}, 123:221--231, 1996.

\bibitem[LZ95]{YYLiZhu95}
Y.~Y. Li and M.~Zhu.
\newblock Uniqueness theorems through the method of moving spheres.
\newblock {\em Duke Math. J.}, 80(2):383--418, 1995.

\bibitem[LZ03]{LiZh03}
Y.~Y. Li and L.~Zhang.
\newblock Liouville-type theorems and {H}arnack-type inequalities for semilinear elliptic equations.
\newblock {\em J. Anal. Math.}, 90:27--87, 2003.

\bibitem[MIT96]{Mitidieri96}
E.~Mitidieri.
\newblock Nonexistence of positive solutions of semilinear elliptic systems in ${R}^{N}$.
\newblock {\em Differ. Integral Equations}, 9:465--480, 1996.

\bibitem[NI82]{Ni82}
W. M.~Ni.
\newblock On the elliptic equation $\Delta u + K(x)u^{(n+2)/(n-2)} = 0$, its generalizations, and applications in geometry.
\newblock {\em Indiana Univ. Math. J.}, 31(4):493--529, 1982.

\bibitem[OB71]{Obata71}
M.~Obata.
\newblock The conjectures on conformal transformations of Riemannian manifolds.
\newblock {\em J. Differential Geometry}, 6:247--258, 1971/1972.

\bibitem[POL11]{Polacik11}
P.~Pol\'{a}\v{c}ik.
\newblock Symmetry of nonnegative solutions of elliptic equations via a result of {S}errin.
\newblock {\em Comm. Partial Differential Equations}, 36(4):657--669, 2011.

\bibitem[PQS07]{PQS07}
P.~Pol{\'a}{\v{c}}ik, P.~Quittner, and Ph. Souplet.
\newblock Singularity and decay estimates in superlinear problems via {L}iouville-type theorems, {I}: {E}lliptic equations and systems.
\newblock {\em Duke Math. J.}, 139(3):555--579, 2007.

\bibitem[PS12]{PhanSouplet12}
Q.~H. Phan and Ph. Souplet.
\newblock Liouville-type theorems and bounds of solutions of {H}ardy--{H}\'{e}non equations.
\newblock {\em J. Differential Equations}, 252:2544--2562, 2012.

\bibitem[QS07]{QS07}
P.~Quittner and Ph. Souplet.
\newblock {\em {\bf Superlinear Parabolic Problems:} Blow-up, Global Existence and Steady States}.
\newblock Birkh\"{a}user Verlag, 2007.

\bibitem[QS12]{QS12a}
P.~Quittner and Ph. Souplet.
\newblock Optimal {L}iouville-type theorems for noncooperative elliptic {S}chr\"{o}dinger systems and applications.
\newblock {\em Comm. Math. Phys.}, 311(1):1--19, 2012.

\bibitem[RZ00]{RZ00}
W.~Reichel and H.~Zou.
\newblock Non-existence results for semilinear cooperative elliptic systems via moving spheres.
\newblock {\em J. Differential Equations}, 161(1):219--243, 2000.

\bibitem[SER71]{Serrin71}
J.~Serrin.
\newblock A symmetry problem in potential theory.
\newblock {\em Arch. Rational Mech. Anal.}, 43:304--318, 1971.

\bibitem[SOU09]{Souplet09}
Ph. Souplet.
\newblock The proof of the {L}ane--{E}mden conjecture in four space dimensions.
\newblock {\em Adv. Math.}, 221(5):1409--1427, 2009.

\bibitem[SOU12]{Souplet12}
Ph. Souplet.
\newblock Liouville-type theorems for elliptic {S}chr\"{o}dinger systems associated with copositive matrices.
\newblock {\em Netw. Heterog. Media}, 7(4):967--988, 2012.

\bibitem[SZ96]{SZ96}
J.~Serrin and H.~Zou.
\newblock Non-existence of positive solutions of {L}ane--{E}mden systems.
\newblock {\em Differ. Integral Equations}, 9(4):635--653, 1996.

\bibitem[SZ98]{SZ98}
J.~Serrin and H.~Zou.
\newblock Existence of positive solutions of the {L}ane--{E}mden system.
\newblock {\em Atti Sem. Mat. Fis. Univ. Modena}, 46:369--380, 1998.
\end{thebibliography}

\end{document}